\documentclass[10pt,a4paper]{amsart}
\usepackage[latin1]{inputenc}
\usepackage{amsmath}
\usepackage{amsfonts}
\usepackage{amssymb}
\usepackage{amsthm}
\usepackage[usenames, dvipsnames]{color}
\usepackage{tikz}
\usepackage[all]{xy}
\usepackage{geometry}

\usepackage{tikz}
\usetikzlibrary{cd,arrows,positioning,calc}
\usetikzlibrary{decorations.markings}
\tikzset{>=stealth}
\tikzcdset{arrow style=tikz}
\tikzset{map/.style={row sep=0em, column sep=0em}}
\tikzset{dot/.style={circle,fill=black,minimum size=5pt,inner sep=0pt}}
\tikzstyle directed=[postaction={decorate,decoration={markings, mark=at position .5 with {\arrow{stealth}}}}] 
\usepackage{tkz-euclide,amsmath}
\usetikzlibrary{calc,intersections,through,backgrounds}
\usepackage{subfig}
\definecolor{liens}{rgb}{1,0,0}
\usepackage[colorlinks=true, linkcolor=blue, 
hyperfootnotes=true,citecolor=blue,urlcolor=black]{hyperref}

\newtheorem*{proposition*}{Proposition~\ref{lem:singcases}}
\newtheorem*{lem**}{Lemma~\ref{lem:doublezero}}
\newtheorem*{thm**}{Theorem~\ref{theo:maintheotransc}}

\newtheorem{thm}{Theorem}[section]
\newtheorem{lemma}[thm]{Lemma}

\newtheorem{prop}[thm]{Proposition}

\newtheorem{defn}[thm]{Definition}
\newtheorem{defi}[thm]{Definition} %this is redundant

\newenvironment{prf}[1]{\trivlist
\item[\hskip \labelsep{\bf #1.\hspace*{.3em}}]}{~\hspace{\fill}~$\square$\endtrivlist}
\newtheorem{ex}[thm]{Example} %this is redundant

\numberwithin{equation}{section}
\def\Z{\mathbb{Z}}
\def\C{\mathbb{C}}
\def\R{\mathbb{R}}
\def\Q{\mathbb{Q}}
\def\N{{\mathbb N}}

\def\P1{\mathbb{P}^{1}}

\def\beq{\begin{equation}}
\def\eeq{\end{equation}}

\def\P2{\mathbb{P}^{2}}
\def\tF{\widetilde{F}}
\def\P1{\mathbb{P}^{1}}

\def\calD{{\mathcal{D}}}

\def\calS{{\mathcal{S}}}

\def\cM{\mathcal{M}}

\def\cL{\mathcal{L}}

\def\Ktld{\widetilde{K}}

\def\calS{\mathcal{S}}
\title{On the $D$-finiteness of generating functions counting \\ small steps  walks in the quadrant }

\author{Charlotte Hardouin}
\address{Institut de Math\'{e}matiques de Toulouse, Universit\'{e} de Toulouse, 118, route de Narbonne, 31062 Toulouse, France and Institut Universitaire de France}
\email{hardouin@math.univ-toulouse.fr}

\begin{document}

 \maketitle
 
\begin{abstract}
The enumeration of small steps walks confined to the first quadrant  of the plane has
attracted a lot of attention over the past fifteen years.  The associated generating functions   are trivariate formal power series in $x,y,t$ where the parameter $t$ encodes the length of the walk while the variables  $x,y$ correspond to the  coordinates of its ending point.  These functions satisfy a functional equation in two catalytic variables.

 Bousquet-M\'{e}lou and Mishna  have associated to any small steps model  an algebraic curve called the kernel curve and   a group called the group of the walk. These two objects turned out to
be central in the classification of  small steps models. In a recent work,  Dreyfus, Elvey Price, and  Raschel  prove that the group of the walk is finite if and only if the generating function is $D$-finite, that is, it satisfies  a linear differential equation with polynomial coefficients in each of its variables $x,y,t$.

  In this paper, we show that if the group of the walk is infinite, the generating  function  doesn't  satisfy a linear differential equation in $x,y$ or $t$  over the field $\Q(x,y,t)$.  The proof of   Dreyfus, Elvey Price, and  Raschel  is based on some singularity analysis. Here, we propose a new strategy  which  relies essentially  on the aforementioned functional equation and on    algebraic arguments. This point of view sheds also a new light on the algebraic nature of the generating functions of small steps models since it relates their   $D$-finiteness more directly  to some geometric properties of the kernel curve.
\end{abstract}

\section{Introduction}
A  small step  model is comprised of a finite set $\mathcal{S} \subset \{-1,0,1\}^2$ of  vectors called the \emph{step set}. A $n$-step walk modeled on $\mathcal{S}$ is a polygonal chain that starts from $(0,0)$ and consists of  $n$ oriented  line segments whose associated translation vectors belong to $\mathcal{S}$. A weighting of $\mathcal{S}$ consists of  a  finite set of non-zero weights $d_{i,j}$ in $\overline{\Q}$ attached to any direction $(i,j)$ in $\mathcal{S}$. One can then  weight   any $n$-step walk   by multiplying all the weights  encountered while traveling this path from $(0,0)$ to its ending point. For a weighted model $\left(\mathcal{S}, (d_{i,j})_{(i,j) \in \mathcal{S}}\right)$, we 
  denote by $q_{i,j}(n)$ the sum of the  weights of all $n$-steps walks  ending at $(i,j)$ and remaining in the first quadrant. The corresponding generating series is the trivariate formal power series $Q(x,y,t)=\sum_{i,j,n} q_{i,j}(n) x^i y^j t^n$.
  
   This  counting problem  is ubiquitous
since lattice walks encode several classes of mathematical objects in
discrete mathematics (permutations, trees, planar maps, \dots), in
statistical physics (magnetism, polymers, \dots), in probability
theory (branching processes, games of chance \dots), in operations
research (birth-death processes, queuing theory). 

The direct computation  of the counting sequence $(q_{i,j}(n))$ being too demanding in general, one natural question is to decide
where $Q(x,y,t)$ fits in the classical hierarchy of power series:
\begin{itemize}
\item  \emph{algebraic:} the series $Q(x,y,t)$  satisfies a nontrivial polynomial relation with coefficients in $\Q(x,y,t)$, 
\item \emph{transcendental and  $D$-finite:} the series $Q(x,y,t)$ is transcendental and, for any derivation $\partial$ in $\{\frac{\partial}{\partial x},\frac{\partial}{\partial y}, \frac{\partial}{\partial t} \}$, it is   $\partial$-finite  over $\Q(x,y,t)$, i.e., there exist $n\in \Z_{\geq 0}$ and  $a_{0},\dots,a_{n}\in \Q (x,y,t)$, not 	all zero, such that 
$$
\displaystyle \sum_{\ell=0}^{n}a_{\ell}\partial^{\ell}Q(x,y,t)=0,
$$

\item \emph{non-D-finite and  differentially algebraic cases:} the series $Q(x,y,t)$ is non-D-finite and, for any derivation $\partial$ in $\{\frac{\partial}{\partial x},\frac{\partial}{\partial y}, \frac{\partial}{\partial t} \}$, it is   $\partial$-algebraic, i.e., 
there exist $n\in \Z_{\geq 0}$ and  a nonzero multivariate polynomial $P_{\partial}\in  \Q(x,y,t)[X_{0},\dots,X_{n}]$, such that 
$$
P_{\partial}(Q(x,y,t),\dots,\partial^{n}Q(x,y,t))=0,$$
 
\item \emph{differentially transcendental cases:} the series is  not $\partial$-algebraic for any derivation $\partial$ in $\{\frac{\partial}{\partial x},\frac{\partial}{\partial y}, \frac{\partial}{\partial t} \}$.
\end{itemize}

   In recent years, the enumeration of such walks has attracted a lot of attention
involving many new methods and tools. All studies concerning the behavior of the generating series $Q(x,y,t)$ begin with the functional equation it satisfies (c.f.~\cite{BMM}).  One first defines a Laurent polynomial $S(x,y) := \sum_{(i,j) \in \calD} d_{i,j} x^iy^j $  called the {\it inventory}  and a polynomial $K(x,y,t) := xy(1-tS(x,y))$  called the {\it kernel polynomial} of the model.  Then,  $Q(x,y,t)$ satisfies
\begin{equation}\label{eq:fnceqn}
K(x,y,t) Q(x,y,t) = xy + F^1(x,t) + F^2(y,t) 
\end{equation}
where  $
F^1(x,t)  := K(x,0,t)Q(x,0,t)+t\epsilon Q(0,0,t),     \,    F^2(y,t):= K(0,y,t)Q(0,y,t)$ and $\epsilon$ is non-zero equal to $1$ if and only if $(-1,-1)$ belong to $\mathcal{S}$.

For \emph{unweighted small steps} walks (that is $\calS \subset
\{-1,0,1\}^2$ and weights all equal to $1$), the 
classification of the generating function is now complete. It
required almost ten years of research and the contribution of many
mathematicians, combining a large variety of tools: elementary power
series algebra \cite{BMM}, computer algebra \cite{KauersBostan},
probability theory \cite{DenisovWachtel}, complex uniformization
\cite{KurkRasch}, Tutte invariants \cite{BBMR16} as well as
differential Galois theory \cite{DHRS}.

In \cite{BMM}, Bousquet-M\'{e}lou and Mishna associated with a non-trivial weighted  model  an algebraic  curve $E$ defined over $\Q(t)$, called the \emph{kernel curve}, of genus $0$ or $1$  and a  group $G$ of  automorphisms of $E$, the \emph{ group of the walk}. These geometric objects play a crucial role in the nature of $Q(x,y,t)$. In \cite{DreyfusHardouinRoquesSingerGenuszero}, the authors prove that the series of a  weighted model associated with a genus zero curve is $\frac{\partial}{\partial x}, \frac{\partial}{\partial y}$-transcendental and \cite{Dreyfushardouintderiv} proves that it is also $\frac{\partial}{\partial t}$-transcendental. For the weighted models associated with a genus one curve and an infinite group, \cite{HardouinSingerSelecta} characterizes those that are $\frac{\partial}{\partial x}, \frac{\partial}{\partial y}$-algebraic in terms of the existence of a \emph{decoupling pair} for the fraction $xy$. The main results of \cite{Dreyfustalg} is that the generating series of a weighted model that is $\frac{\partial}{\partial x}, \frac{\partial}{\partial y}$-algebraic is also $\frac{\partial}{\partial t}$-algebraic. Finally, in  \cite{DEPR}, the authors  prove that  the generating series of a weighted model is algebraic if and only if its group is finite and the fraction $xy$ has a decoupling pair and $D$-finite if and only if the group is finite.

In this paper, we prove that if the group associated to a weighted model  is infinite then the generating series is not $\partial_x,\partial_y$-finite nor $\partial_t$-finite (see Theorem~\ref{thm:main} below). Unlike the proof of \cite{DEPR}, our arguments do not rely on a singularity analysis of the generating series but on a more algebraic and geometric approach. We use the non-Archimedean uniformization of the kernel curve developed in \cite{Dreyfushardouintderiv} to meromorphically continue the generating series  over a non-Archimedean field extension $C$ of $\Q(t)$. This  meromorphic continuation $F(s)$ satisfies a $q_2$-difference  equation of the form $F(q_2s)=F(s)+b(s)$ for some $b(s)$ in $C_{q_1}$, the field of $q_1$-elliptic functions over the   Tate curve $C^*/{q_1^{\Z}}$. In that setting, one can derive the function $F(s)$ with respect to $s$ and $t$ in a compatible way with its $q_2$-difference equation and relate the $D$-finitness of the generating series to the $\partial_s,\partial_t$-finitness of $F(s)$. Moreover, the group of the walk is infinite if and only if the Tate curves $C^*/{q_1^{\Z}}$ and $C^*/{q_2^{\Z}}$ are  not isogenous which  implies  that the fields $C_{q_1}$ and $C_{q_2}$ are linearly disjoint over $C$. Our proof  relies essentially on this last characterization and on basic properties of difference and differential equations.  

The paper is organized as follows. In Section 2, we recall some basic notions on walks as well as former results on their algebraic classification. We reintroduce also briefly the non-Archimedean framework of \cite{Dreyfushardouintderiv}. Section 3 is devoted to the proof of Theorem~\ref{thm:main}. Appendix~A contains some elementary computations on the poles of a potential decoupling of $xy$.

\section{Preliminaries on mall steps walks}\label{subsect:preliminaries}

\subsection{The kernel curve}\label{subsec:kernelcurve}
 
 \par The \emph{kernel} polynomial is  a bivariate polynomial in $x,y$ with coefficients in $\Q(t)$ considered here as  a  valued field  endowed with the valuation at $t$ equal zero. The field $\Q(t)$   is neither  algebraically closed   nor complete and  the field of Puiseux series with coefficients in $\overline{\Q}$ is algebraically closed but not complete.
\par  Therefore,  we  consider here  the field $C$ of Hahn series or  Malcev-Neumann series with coefficients in  $\overline{\Q}$, and monomials from $\Q$. We recall that  a  Hahn series $f$
is a formal power series $\sum_{ \gamma \in \Q }c_\gamma t^\gamma$ with coefficients $c_\gamma$ in $\overline{\Q}$ and such that the subset $\{\gamma | c_\gamma \neq 0 \}$ is a well ordered subset of $\Q$.    The valuation $v_0( f)$ of the Hahn series $f=\sum_{ \gamma \in \Q }c_\gamma t^\gamma \in C$ is the smallest element of the subset $\{\gamma | c_\gamma \neq 0 \}$. 
The field $C$ is algebraically closed by \cite[Theorem 1]{MacLane} and spherically complete 
with respect to the valuation at zero and thereby complete  (see \cite[Corollaries 2.2.7 and  3.2.9]{AschenbrennerVandenDriesVanDerHoeven}). Let us fix, once for all, $\alpha \in \R$ such that $0< \alpha <1$. For any $f \in C$, we define the norm of $f$ as $|f|= \alpha^{v_0(f)}$.

We recall that $\P1(C)$ is the projective line over $C$. It consists in  the equivalence classes of pairs $(\alpha_0,\alpha_1)$ of elements of $C$ that are both non-zero modulo the equivalence relation $(\alpha_0,\alpha_1) \sim (\alpha_0',\alpha_1') $ if and only if 
$\alpha_0=\lambda \alpha_0'$ and $\alpha_1=\lambda \alpha_1'$ for some non-zero $\lambda$ in $C$. The equivalence class of $(\alpha_0,\alpha_1)$ is denoted by $[\alpha_0:\alpha_1]$. To shorten notation, one shall sometimes  denote by $\alpha$ the class $[\alpha:1]$ and by $\infty$ the class $[1:0]$. This allows us to write  $\P1(C)$ as  $C \cup \{\infty\}$.

Fixing a weighted model $(\mathcal{S},(d_{i,j})_{(i,j) \in \mathcal{S}})$. The associated  kernel curve $E$   is defined by 
\[E =\{( [\alpha_0:\alpha_1], [\beta_0,\beta_1]) \in \P1(C)\times \P1(C) |   \Ktld( \alpha_0, \alpha_1, \beta_0,\beta_1 )=0 \}\]
with 
$
\widetilde{K}(x_0,x_1,y_0,y_1,t)= x_0x_1y_0y_1 -t \sum_{(i,j) \in \mathcal{S}} d_{i-1,j-1} x_0^{i} x_1^{2-i}y_0^j y_1^{2-j}={x_1^2y_1^2K\left(\frac{x_0}{x_1},\frac{y_0}{y_1},t\right)}. 
$

In \cite{DHRSkernel}, the authors   fix a transcendental value of $t$ in $\C$ and define the kernel curve as an algebraic curve in $\P1(\C) \times \P1(\C)$. We want to stress out that the geometric properties, smoothness, irreducibility and genus, do not depend on the algebraically closed  extension of $\Q(t)$  in which one considers the zeroes of  the kernel polynomial. Therefore, Proposition 2.2 in \cite{DHRSkernel} classifies the step sets $\mathcal{S}$ that yield to a reducible kernel curve $E$ and, more generally, to one dimensional models. For these models, the generating series is algebraic so that we restrict our study to models associated with an irreducible  kernel curve. For these models,  the kernel polynomial is of degree $2$ in each of its variables and $E$ is of genus $0$  if $E$ has a singular point and of genus  $1$ if $E$ is a smooth curve. Moreover,  Corollary 3.6 in \cite{DHRSkernel} shows that the genus of $E$ is entirely determined by the  step set of the model.

The  function field $C(E)$ of the kernel curve $E$ is the fraction field of $C[x,y]$ mod out by the ideal generated by $K(x,y)$. It can also be identified with the quotient of the ring of \emph{regular fractions}, that is, the elements of $C(x,y)$ whose denominator is not divisible  by $K$, modulo the ideal generated by $K(x,y)$.  Since $K(x,y)$ has degree $2$ in $x$ and $y$,  the field  $C(E)$ is a Galois field extension of degree $2$ of $C(x)$ and of $C(y)$. Denote respectively by $\iota_1$ (resp. $\iota_2$) the non-trivial Galois automorphism of $C(E)$ over $C(x)$ (resp. over $C(y)$). These field automorphisms corresponds to the two involutive automorphisms of $E$, still denoted by $\iota_1, \iota_2$, corresponding to the automorphisms of the two   covers  given by $E \rightarrow \P1, ([x_0:x_1], [y_0:y_1] )  \mapsto [x_0:x_1]$ and $E \rightarrow \P1, ([x_0:x_1], [y_0:y_1] )  \mapsto [y_0:y_1]$.  The automorphism of the walk is $\iota_1 \circ \iota_2$ and the  group of the walk is defined as the group of automorphisms of $C(E)$ generated by $\iota_1$ and $\iota_2$. It is a dihedral group  whose order is finite if and only if $\iota_1 \circ \iota_2$ is of finite order. 

\subsection{Differential transcendence results}

We have the  following result.

\begin{prop}[ Theorem  in \cite{DreyfusHardouinRoquesSingerGenuszero2} and   Theorem 2 and Theorem 4.14 in \cite{Dreyfushardouintderiv} ]\label{prop:tdiffalgimpliesgenusone}

For any weighted model with genus zero kernel curve, the generating series is $x,y,t$-differentially transcendental. For any weighted model with genus one kernel curve, if the generating series is $x$-differentially transcendental then it is $t$-differentially transcendental. 
\end{prop}

When the group of the walk is infinite and the genus of the curve $E$ is one, the  differential and algebraic nature of the generating series  is controlled by the decoupling of a certain rational fraction.  
\begin{defn}
Let $K(x,y)$ be an irreducible polynomial in $\Q[x,y,t]$.  We say that a  regular fraction $r(x,y)$ decouples or admits a \emph{decoupling} if there exist $f(x)$ in $\Q(x,t)$, $g(y) \in \Q(y,t)$   and $h(x,y)$  a regular  fraction such  that 
\begin{equation}\label{eq:decouplingdef}
r(x,y)= f(x) +g(y) +K(x,y) h(x,y).
\end{equation} 
\end{defn}
Note that when the group of the walk is infinite, any  decoupling $(f(x),g(y))$ of a regular fraction $r(x,y)$ differ by a constant pair of the form $(c,-c) $ in $\Q(t)^2$. Indeed, the difference $(\alpha(x),\beta(y))$ of two decoupling pairs is an invariant in the sense of Definition 4.3 in \cite{BBMR16} and is therefore constant when the group is infinite by  \cite[Theorem 4.6]{BBMR16}.  The following proposition characterizes the genus one models whose generating series is $x,y$-differentially algebraic. 
\begin{prop}[Theorem 3.8 and Proposition 3.9 in \cite{HardouinSingerSelecta}]\label{prop:xdiffalgequivalentdecoupled}
For any weighted model with genus one kernel curve, the generating series is $x,y$-differential algebraic if and only if $xy$ decouples.
\end{prop}

Therefore, according to the above propositions, if the generating series $Q(x,y,t)$ is $\partial_t$-finite then the kernel curve is of genus one, the generating series $Q(x,y,t)$ is $x,y$-differentially algebraic and $xy$ decouples.  Thus in the rest of the paper, we only consider    weighted models $\left(\mathcal{S},(d_{i,j})_{(i,j) \in \mathcal{S}}\right)s$ whose  kernel curve is of genus one.

\subsection{Non-Archimedean framework}

 In order to study the  $t$-derivatives of  $Q(x,y,t)$, we do not assign to $t$ a fixed value in $\C$ as it is done in \cite{dreyrasch} but we proceed as in \cite{Dreyfushardouintderiv}. This strategy requires to work in a non-Archemedean framework. In this section,  we recall some basic notions from the function theory  over non-Archemedean fields. We refer to \cite{Guntzer} and  \cite{roquette1970analytic} for a brief introduction on these notions and to \cite{BGRrigid} for a complete presentation.

\begin{defi}
We consider formal Laurent series $f(s)=\sum_{n \in \Z} a_n s^n$ such that $f(\alpha)$ is convergent for every element $\alpha \neq 0$ in $C$ and we call such a function a \emph{holomorphic function on $C^*$}. The ring $\mathcal{H} ol(C^*)$ of holomorphic functions form an integral domain whose fraction field is denoted  $\cM er(C^*)$ and called the field  of \emph{meromorphic functions over $C^*$}. 
\end{defi}

Analogously to meromorphic functions over $\C$, one can define the notion of poles and zeroes in $C^*$ of a meromorphic functions over $C^*$  as well as the order of these zeroes and poles (see \cite[Page 11]{roquette1970analytic}). A  meromorphic function on $C^*$ with no zeroes nor poles on $C^*$ is of the form $c s^m$ for some $c$ in $C$ and $m$ in $\Z$ (\cite[Korollar 1]{Guntzer}). 

The field $\cM er(C^*)$ comes naturally with a derivation $\partial_s=s\frac{d}{ds}$. It is easily seen that $\cM er (C^*)^{\partial_s}$ the field of meromorphic functions whose derivative, with respect to $\partial_s$, vanishes coincides with $C$. However, the field $C$  carries a non trivial derivation   $\partial_t$ defined   as follows
$$ \partial_t\left(\sum_{ \gamma \in \Q }c_\gamma t^\gamma \right)= \sum_{ \gamma \in \Q }c_\gamma \gamma t^\gamma.$$
The derivation  $\partial_t$ extends the derivation $t\frac{d}{dt}$ of $\Q(t)$ and  is continuous (see \cite[Example (2), \S 4.4]{AschenbrennerVandenDriesVanDerHoeven}). This last property allows us to extend the action of $\partial_t$ to $\cM er(C^*)$.

\begin{lemma}
The map $\partial_t$  on $\mathcal{H} ol(C^*)$  defined by 
\[\partial_t( \sum_{n \in \Z} a_n s^n)= \sum_{n \in \Z} \partial_t(a_n) s^n \]
is well defined. It  is a derivation of  $\mathcal{H} ol(C^*)$ which  extends to a derivation of $\cM er(C^*)$ that  we still denote by $\partial_t$.
\end{lemma}
\begin{proof}
We claim that, for  any  holomorphic function $f$ and $\alpha$ in $C^*$, the series $f(\alpha)=\sum_{n \in \Z} \partial_t(a_n) \alpha^n$ is convergent. Indeed, for any $r \in \Q $,  we have  $\partial_t(a_nr^n)=\partial(a_t)r^n$. Moreover, since $f(s)$ is holomorphic over $C^*$, we have $\lim_{|n| \rightarrow + \infty} |a_n r^n|=0$. Thus, the continuity of $\partial_t$ on $C$ yields 
\[ \lim_{n \rightarrow  + \infty } |\partial_t(a_n) | r^n =  0. \]  Thus, $\partial_t$ is  well defined on  $\mathcal{H} ol(C^*)$. It is easily seen that it is a derivation of the ring of holomorphic functions so that it extends uniquely  to its fraction field $\cM er(C^*)$.
\end{proof}

In the sequel, we shall also consider $q$-difference operators acting on $\cM er(C^*)$.  For $q$ an element of $C$ such that $0 < |q| <1$, the $q$-difference operator $\sigma_q$ is the  automorphism of $\cM er(C^*)$ defined by  $\sigma_q(f(s))=f(qs)$ for any $f$ in $\cM er(C^*)$. Its field of constants is   $C_q=\{ f | \sigma_q(f)=f \}$ and its elements are called $q$-elliptic functions since they are the multiplicative analogue of the elliptic functions over $\C$.  Whereas the derivation $\partial_s$ commutes with $\sigma_q$, it is not the case for $\partial_t$.  
The following lemma shows how one can construct a derivation that depends on the derivations $\partial_s$ and $\partial_t$ and commutes with $\sigma_q$.

\begin{lemma}\label{lemma:tderivationdfinite}
Let $q$ be an element of $C$ such that $0< |q| <1$ and $f \in C_q \setminus C$. The following hold
\begin{enumerate}
\item $ \sigma_q( \frac{\partial_t f}{\partial_s f}) = \frac{\partial_t f}{\partial_s f} -\frac{\partial_t (q)}{q}$,
\item the application $\Delta_{q,f}= \partial_t -\frac{\partial_t f}{\partial_s f} \partial_s$ is a derivation of $\cM er(C^*)$ that commutes with $\sigma_q$.  
\end{enumerate}
When the context is clear, we write $\Delta_q$ instead of  $\Delta_{q,f}$.
 \end{lemma}
\begin{proof}
We have $\sigma_q(\partial_s g) =\partial_s (\sigma_q(g))$ for any function $g$ meromorphic over $C^*$.  Moreover,  we have
\[ \partial_t (\sigma_q(g))= \frac{\partial_t (q)}{q} (\partial_s g) (qs) +(\partial_t g ) (qs ), \]
that is 
\[ \sigma_q \partial_t= \partial_t \sigma_q  -\frac{\partial_t (q)}{q} \sigma_q \partial_s. \]
Thus, we conclude  from $f(qs)= f(s)$ that $\sigma_q(\partial_s f)=\partial_s f$ and that 
\[ \sigma_q( \frac{\partial_t f}{\partial_s f}) = \frac{\partial_t f}{\partial_s f} -\frac{\partial_t (q)}{q}.  \]

Let $g$ be in $\cM er(C^*)$. Using the above formulas, we  compute 
\begin{eqnarray*}
\sigma_q(\Delta_{q,f}(g)) & = &  \sigma_q \partial_t(g) - \frac{\partial_t f}{\partial_s f} \sigma_q (\partial_s g) + \frac{\partial_t (q)}{q} \sigma_q \partial_s (g), \\
\sigma_q( \Delta_{q,f} (g) )& =& \partial_t \sigma_q(g)  -\frac{\partial_t (q)}{q}  \partial_s \sigma_q (g) - \frac{\partial_t f}{\partial_s f} \sigma_q (\partial_s g) + \frac{\partial_t (q)}{q} \sigma_q \partial_s (g), \\
\sigma_q ( \Delta_{q,f} (g)) & =& \Delta_{q,f} ( \sigma_q (g)).
\end{eqnarray*} 
\end{proof}

The following lemma is crucial  to prove the non-D-finitness of the generating series when the group is infinite. It  shows that the function fields of  two non-isogenous elliptic curves defined over $C$ are linearly independent over $C$. 
\begin{lemma}\label{lem:linindepnonisogenous}
Let $q_1, q_2$ be two elements of $C$ such that  $0< |q_1|, |q_2| <1$. If  $q_1,q_2$ are multiplicatively independent that is, there is no $(r,l) \in \Z^2 \setminus (0,0)$ such that $q_1^r=q_2^l$ then  the subfields $C_{q_1} $ and $C_{q_2}$ of $\cM er(C^*)$ are linearly disjoint over $C$\footnote{We recall that two subfields $L,M$ of a field $K$ are linearly disjoint over a field $k \subset L\cap M$ if any $k$-linearly independent elements of $L$ remain linearly independent over $M$.}.
\end{lemma}
\begin{proof}
We claim that  $C_{q_1} \cap C_{q_2}=C$. Indeed, let $f$ be in $C_{q_1} \cap C_{q_2}$. The set $S$ of non-zero poles of $f$  is finite modulo  $q_1^{\Z}$ (see \cite[\S 2]{roquette1970analytic}). Here, we say that $\alpha= \beta$ modulo $q_1^{\Z}$ if $\alpha= q_1^r \beta$ for some $r$ in $\Z$.  Since $\sigma_{q_2}(f)=f$, the set $S$ must be stable by $\sigma_{q_2}$. If $S \neq \emptyset$ then there exists $\alpha \in S$ and $r\in \Z^*$ such that $q_2^r \alpha= q_1^l \alpha$ for some integer $l$.  Since $\alpha$ is non-zero, this contradicts the multiplicative independence of $q_1$ and $q_2$. Thus, $f$ has no non-zero poles. Reasoning analogously  with $1/f$, we conclude that $f$ has no non-zero poles or zeroes so that $f=cs^d$ for some $c$ in $C$ and $d$ in $\Z$. Since $\sigma_{q_1}(f)=f$ and $q_1,q_2$ are multiplicatively independent, we conclude that $f=c$. 

Suppose to the contrary that the fields are not linearly disjoint over $C$.  Choose  a minimal $C_{q_2}$-linear relation among  elements  $f_1,\dots,f_r$ of $C_{q_1}$ that are $C$-linearly independent. Without loss of generalities, one can assume that this minimal relation is of the form 
\begin{equation}\label{eq:minimalliaisonnonisogeneous}
f_1 + \lambda_2 f_2 + \dots + \lambda_r f_r =0,
\end{equation}  
for $\lambda_2, \dots, \lambda_r  \in C_{q_2}$. Now applying $\sigma_{q_1}$ to \eqref{eq:minimalliaisonnonisogeneous} and subtracting \eqref{eq:minimalliaisonnonisogeneous}, we find 
\begin{equation}\label{eq:minimalliaisonnonisogeneous2}
 (\sigma_{q_1}( \lambda_2) -\lambda_2)  f_2 + \dots +  (\sigma_{q_1}( \lambda_r) -\lambda_r)  f_r =0.
\end{equation} 
The minimality of \eqref{eq:minimalliaisonnonisogeneous} yields  $\sigma_{q_1}(\lambda_2)-\lambda_2=\dots = \sigma_{q_1}( \lambda_r) -\lambda_r =0$. Thus, $\lambda_i \in C_{q_1} \cap C_{q_2}=C$  for all $i=2,\dots,r$. A contradiction with the $C$-linear independence of $f_1,\dots,f_r$.
\end{proof}

\subsection{A functional equation over a Tate curve}

In this section, we consider a weighted model with a genus one kernel curve $E$. In \cite{Dreyfushardouintderiv}, the authors use rigid analytic geometry to associate to the generating series a difference equation whose coefficients are $q$-elliptic for a certain $q$. We  recall their results in the following propositions. 

\begin{prop}[Theorem 4.3 in \cite{Dreyfushardouintderiv}]\label{prop:uniformizationcurve}
There exists an element $q_1$ in $C$ such that $ 0 < |q_1|<1$   and two $q_1$-elliptic functions $x(s),y(s)$ such that $K(x(s),y(s))$ vanishes for all $s$ in $C^*$ and the map 
\[ \phi : C^* \rightarrow E, s \mapsto (x(s),y(s)) \]
is surjective. The application $\phi$ yields a field isomorphism between $C(E)$ and $C_{q_1}$: 
\[ C(E) \rightarrow C_{q_1}, f(x,y) \mapsto \widetilde{f}(s)=f(x(s),y(s)).\]
 Moreover, there exists a non-zero element $q_2$ in $C$ such that $|q_2| \neq 1$ and the group of the walk is infinite if and only if $q_1$ and $q_2$ are muliplicatively independent.
\end{prop}
This proposition allows to identify $E$ with $C^*/q_1^{\Z}$ which is a multiplicative version  of the quotient of $\C$ by a lattice. The action of the automorphism $\iota_1\circ \iota_2$ on $E$ lifts as the action of $\sigma_{q_2}$ on $C^*$. The strategy of \cite{Dreyfushardouintderiv} is to  use the above parametrization of $E$,  the zero locus of $K(x,y,t)$ together with the functional equation \eqref{eq:fnceqn} to meromorphically prolong $F^1(x,t)$ and $F^2(y,t)$ so that they satisfy $q_2$-difference equations of order $1$. 
We summarize the results of Section 4.3 in \cite{Dreyfushardouintderiv} in particular Proposition 4.5, Lemma 4.7, Lemma 4.10, Theorem 4.12 and the remark after Corollary 4.16 in the following proposition.

\begin{prop}\label{prop:functionalequandcaracterization}
There exist two   non-empty  annuli $\mathcal{U}_x, \mathcal{U}_y \subset C^*$  and two meromorphic functions $\widetilde{F}^1(s)$ and $\widetilde{F}^2(s)$ over $C^*$ such that  \begin{itemize}
\item $\phi(\mathcal{U}_x)=\{ (x,y) \in E | |x | \leq 1\}$ and  $\phi(\mathcal{U}_y)=\{ (x,y) \in E | |y| \leq 1\}$,
\item $\tF^1(s)=F^1(x(s),t)$ for any $s$ in $\mathcal{U}_x$ and $\tF^2(s)=F^2(y(s),t)$ for any $s \in \mathcal{U}_y$,
\item $\sigma_{q_2}(\tF^1)=\tF^1 + b_1$ and $\sigma_{q_2}(\tF^2)=\tF^2 + b_2$ where $b_1=\left(\sigma_{q_2}(x(s))-x(s) \right)\sigma_{q_2}(y(s))$     $b_2=\left(\sigma_{q_2}(y(s))-y(s) \right)x(s)$  are elements of $C_{q_1}$,
\item for all $s$, we have $\widetilde{F}^1(s) +\widetilde{F}^2(s) =x(s)y(s)$,
\item $b_2= \sigma_{q_2}(\widetilde{g}) -\widetilde{g}$ for some $\widetilde{g} \in C_{q_1}$ if and only $xy$ decouples. In that case, $\widetilde{g}=g(y(s))$ where  $xy=f(x)+g(y) +K(x,y)h$ is a decoupling of $xy$. 
\end{itemize}
\end{prop}

For ease of notation, let us denote by $\Delta_{q_1}$ the derivation  defined in Lemma  \ref{lemma:tderivationdfinite} with $f=y(s)$. The following lemma relates the $\partial_y$ and $\partial_t$-finiteness of $F^2(y,t)$ over $\Q(y,t)$ to the $\Delta_{q_1}$ and $\partial_s$-finiteness of $\tF^2(s)$ over $C_{q_1}$. An analogous result holds for $F^{1}(x,t)$ and $\tF^1(s)$.

\begin{lemma}\label{lem:dfinitenessseriesanduniformization}
If $F^2(y,t)$ is $\partial_y$-finite over $\Q(y,t)$ then $\tF^2(s)$ is $\partial_s$-finite over $C_{q_1}$.

If $F^2(y,t)$ is $\partial_t$-finite over $\Q(y,t)$ then $\tF^2(s)$ is $\Delta_{q_1}$-finite over $C_{q_1}$.
\end{lemma}
\begin{proof}
For $s \in \mathcal{U}_y$, we derive $\tF^2(s)=F^2(y(s),t)$ with respect to $\partial_s$ and   find that 
\begin{equation}\label{eq:ysderivation}\partial_s (\tF^2(s))= \partial_s(y(s)) (\frac{\partial F^2}{ \partial y}) (y(s),t) .\end{equation}
Since $\phi$ is a parametrization of $E$, the function $y(s)$ is transcendental over $C$ and $\partial_s(y(s))$ is non-zero. Moreover, since $\partial_s$ and $\sigma_{q_1}$ commute, the field $C_{q_1}$ is stable under the derivation $\partial_s$ so that $\partial_s^i(y(s)$ belong to $C_{q_1}$ for any integer $i$.  Deriving again, we find that 
\[( \frac{\partial^k F^2}{ \partial y^k} )(y(s),t) =\frac{1}{\partial_s(y(s))^k}\partial_s^k(\tF^2(s))+\lambda_{k-1}(s)\partial_s^{k-1}(\tF^2(s))+\dots+\lambda_0(s) \tF^2(s),\]
for some $\lambda_0,\dots,\lambda_{k-1}$ in $C_{q_1}$. This computation  proves the first claim. 

For $s \in \mathcal{U}_y$, we derive $\tF^2(s)=F^2(y(s),t)$ with respect to $\partial_t$ and   find that 
\[ \partial_t( \tF^2(s))= \partial_t(y(s)) (\frac{\partial F^2}{\partial y})(y(s),t) + (t\frac{\partial F^2}{\partial t} )(y(s),t).\]
Combining   this equation with \eqref{eq:ysderivation},  we obtain 
\[ (t\frac{\partial F^2}{\partial t} )(y(s),t) = \partial_t (\tF^2(s)) -\frac{\partial_t(y(s))}{\partial_s(y(s))} \partial_s(\tF^2(s)),\]
that is, $ (t\frac{\partial F^2}{\partial t} )(y(s),t)= \Delta_{q_1} (\tF^2(s))$ for any $s \in \mathcal{U}_y$. One can derive again and find that  $ (\partial_t^k F^2)(y(s),t)= \Delta_{q_1}^k (\tF^2(s))$ for any non-negative integer $k$. If $F^2(y,t)$ is $\partial_t$-finite over $\Q(y,t)$, then  there exists $c_0(y,t),\dots,c_{r-1}(y,t)$ in $\Q(y,t)$ such that 
\[c_0(y,t) F^2(y,t) + \dots + (t \frac{\partial}{\partial t})^r (F^2(y,t))=0 .\]
Evaluating this equation at $y=y(s)$ for $s \in \mathcal{U}_y$, one finds that 
\begin{equation} \label{eq:deltaqfiniteU} c_0(y(s),t) \tF^2(s) + \dots +  \Delta_{q_1}^r (\tF^2(s))=0 \end{equation}	
for any $s$ in $\mathcal{U}_y$. Since $y(s) \in C_{q_1}$, the functions $c_i(y(s),t)$ all belong to $C_{q_1}=C(x(s),y(s))$. 
By the principal of isolate zeroes (\cite[p 11]{roquette1970analytic}),  we   conclude that \eqref{eq:deltaqfiniteU} is valid on the whole $C^*$ and that  $\tF^2(s)$ is $\Delta_{q_1}$-finite over $C_{q_1}$. 
\end{proof}

\section{D-finiteness and group of the walk}

We first prove that a model with an infinite group of the walk can not have a $\partial_y$-finite generating series.
\begin{lemma}\label{lemma:sdfiniteimpleisfinitegroup}
If $F^2(y,t)$ is $\partial_y$-finite over $\Q(y,t)$ then the group of the walk is finite.
\end{lemma}
\begin{proof}
By Proposition~\ref{prop:tdiffalgimpliesgenusone}, we conclude that the genus of the kernel curve is one. Proposition~\ref{prop:xdiffalgequivalentdecoupled} implies that  the fraction $xy$  has a decoupling, that is, \[xy= f(x) +g(y) +K(x,y,t)h(x,y),\]
with $f(x) \in \Q(x,t), g(y) \in \Q(y,t)$ and $h(x,y)$ a regular fraction.  Therefore, $x(s)y(s)=\widetilde{f}(s) +\widetilde{g}(s)$   where  $\widetilde{f}(s)=f(x(s)),\widetilde{g}= g(y(s))$ belong to $C_{q_1}$ by Proposition~\ref{prop:uniformizationcurve}.  By Lemma~\ref{lem:dfinitenessseriesanduniformization},  if $F^2(y,t)$ is $\partial_y$-finite  over $\Q(y,t)$  then $\tF^2$ is $\partial_s$-finite over $C_{q_1}$. By Proposition \ref{prop:functionalequandcaracterization}, we have $\sigma_{q_2}(\tF^2(s) -\widetilde{g})=\tF^2(s) -\widetilde{g}$. This means that there exists $e \in C_{q_2}$ such that $\tF^2(s)= e(s) +\widetilde{g}(s)$.  Since $\tF^2(s)$ is $\partial_s$-finite over $C_{q_1}$ and $\widetilde{g}(s) \in C_{q_1}$, we find that  $e(s)$ is also $\partial_s$-finite over $C_{q_1}$.  

Suppose to the contrary that the group of the walk is infinite. By proposition \ref{prop:functionalequandcaracterization}, the elements $q_1$ and $q_2$ are multiplicatively independent. Thus, the fields $C_{q_1}$ and $C_{q_2}$ are linearly independent over $C$ by Lemma \ref{lem:linindepnonisogenous}. The functions $\partial_s^k(e)$ all belong to $C_{q_2}$. Since $e(s)$ is $\partial_s$-finite over $C_{q_1}$,  they are $C_{q_1}$-linearly dependent  and therefore $C$-linearly dependent by linear disjointness of the fields $C_{q_1}$ and $C_{q_2}$ over $C$. We conclude that $e(s)$ is $\partial_s$-finite over $C$.  Therefore,   $e(s)$ has no pole in $C^*$.  Since any $q_2$-elliptic function with no poles in $C^*$ must be a constant \footnote{This can be deduced  from  the fact that the field of $q_2$-elliptic functions is a function field of genus $1$ over the algebraically closed field $C$ by \cite[IV, page 19]{roquette1970analytic}. A $q_2$-elliptic function with no poles in $C^*$ therefore corresponds to an element of the function field with no poles and is a constant by \cite[Corollary I.1.19]{Stichtenothalgfunctionfields}.},  the element $e(s)$ in $C_{q_2}$ must be  constant, equal to say  $e$ in $C$. Since for any $s$, we have 
\[\tF^1(s) +\tF^2(s) = \widetilde{f}(s) +\widetilde{g}(s), \]
we conclude that $\tF^1(s)= \widetilde{f}(s)-e$.  Proposition~\ref{prop:functionalequandcaracterization} implies that  
$F^2(y,t)=g(y)+e $  for any  $ y$  such that  $ |y | \leq 1$  and $ F^1(x,t) =f(x) +e $ for any   $x$   such that  $ |x| \leq 1. $
 Lemma~\ref{lem:polesdecoupling} asserts that either $f(x)$  has a pole such that $|x| \leq 1$ or $g(y)$ has a pole such that $|y| \leq 1$. This is a contradiction  since $F^1(x,t)$ and $F^2(y,t)$ are  respectively analytic on $|x| \leq 1$  and $|y| \leq 1$.
\end{proof}

The following theorem proves that the generating series of a model with infinite group of the walk can not be $\partial_t$-finite.  
\begin{thm}\label{thm:tdf}
If $F^2(y,t)$ is $\partial_t$-finite over $\Q(y,t)$ then the group of the walk is finite.
\end{thm}
\begin{proof}
By  Proposition \ref{prop:tdiffalgimpliesgenusone}, we conclude that the genus of the kernel curve is one and that $F^2(y,t)$ is $y$-differentially algebraic. Proposition \ref{prop:xdiffalgequivalentdecoupled} implies that  the fraction $xy$  has a decoupling, that is, \[xy= f(x) +g(y) +K(x,y,t)h(x,y),\]
with $f(x) \in \Q(x,t), g(y) \in \Q(y,t)$ and $h(x,y)$ a regular fraction.  By Proposition \ref{prop:functionalequandcaracterization}, we have $\sigma_{q_2}(\tF^2(s) -\widetilde{g}(s))=\tF^2(s) -\widetilde{g}(s)$  where  $\widetilde{g}= g(y(s))$ is  in $C_{q_1}$. This means that there exists $e(s) \in C_{q_2}$ such that $\tF^2(s)= e(s) +\widetilde{g}(s) $. 

Suppose to the contrary that the group of the walk is infinite so that $C_{q_1}$ and $C_{q_2}$ are linearly disjoint over $C$. 
By Lemma \ref{lem:dfinitenessseriesanduniformization}, the function $\tF^2(s)$ is $\Delta_{q_1}$-finite over $C_{q_1}$. Since $\widetilde{g}(s)$ is in $C_{q_1}$, the function $e(s)$ is also $\Delta_{q_1}$-finite over $C_{q_1}$. Note that, unlike the derivatives with respect to $\partial_s$,  the derivatives of $e(s)$ with respect to $\Delta_{q_1}$ do not necessarily belong to $C_{q_2}$. 

Let $\cL$ be the linear differential operator in $C_{q_1}[\Delta_{q_1}]$ such that $\cL(e)=0$ and let $r$ be the order of the linear differential operator $\cL$. Set  $\mathcal{C} =\{h \in \cM er(C^*) | \Delta_{q_1}(h)=0\}$ and  $V=\{ h \in \cM er(C^*) | \cL(h)=0 \}$.   By \cite[Lemma 1.10]{VdPS97}, the dimension of $V$ over $\mathcal{C} $ is bounded by $r$. Since $\sigma_{q_1} \cL= \cL (\sigma_{q_1})$, we find that  $\sigma_{q_1}^k(e)$ belong to $V$ for any non-negative integer $k$. This proves that the family $(\sigma_{q_1}^k(e))_{k \in \N}$ is  $\mathcal{C} $-linearly dependent. 

Let 
\begin{equation}\label{eq:minimalliaison3}
c_0 e + \dots + c_{r-1} \sigma_{q_1}^{r-1}(e) + \sigma_{q_1}^r(e)=0,
\end{equation} 
be a minimal $\mathcal{C} $-linear relation for the  family $(\sigma_{q_1}^k(e))_{k \in \N}$. Since $\sigma_{q_2}$ and $\sigma_{q_1}$ commute and $\sigma_{q_2}(e)=e$, we apply $\sigma_{q_2}$ to \eqref{eq:minimalliaison3} and subtracting \eqref{eq:minimalliaison3}, we obtain 
\begin{equation}\label{eq:minimalliaison4}
(\sigma_{q_2}(c_0)- c_0)  e + \dots + (\sigma_{q_2}(c_{r-1})-c_{r-1}) \sigma_{q_1}^{r-1}(e) =0.
\end{equation} 
By minimality, we have $\sigma_{q_2}(c_i)=c_i$ for any $i=0,\dots,r-1$. 
Moreover $\Delta_{q_1}(c_i)=0$ which gives $\partial_t(c_i)= \frac{\partial_t y(s)}{\partial_s y(s)} \partial_s (c_i)$.

 We claim  that  the  $c_i$'s are all in $C$. Suppose to the contrary that some $c_i$ is not in $C$  so that  $\partial_s(c_i) \neq 0$ and 
\[ \frac{\partial_t(c_i)}{\partial_s(c_i)} =\frac{\partial_t y(s)}{\partial_s y(s)}. \]
Denote by  $\alpha$ the  function  $\frac{\partial_t(c_i)}{\partial_s(c_i)}$.  Lemma \ref{lemma:tderivationdfinite} implies that 
\begin{itemize}
\item $\sigma_{q_2} ( \alpha)= \alpha -\frac{\partial_t (q_2)}{q_2}$ because $c_i \in C_{q_2}$,
\item and $\sigma_{q_1} (\alpha)= \alpha - \frac{\partial_t (q_1)}{q_1}$ because $y(s) \in C_{q_1}$.
\end{itemize}
Deriving these equations with respect to $\partial_s$, we find  that $\partial_s(\alpha) \in C_{q_1} \cap C_{q_2} =C$. Therefore, there exists $c \in C$ such that $\partial_s(\alpha)=c$. This implies that $\alpha$ has no  poles in $C^*$ and is an holomorphic function over $C^*$  by \cite[Korollar 1]{Guntzer}. It is then easily seen that  we must have $c=0$ and $\partial_s(\alpha)=0$. This implies that $\alpha \in C$ which yields a contradiction with the fact that $\sigma_{q_1}(\alpha)- \alpha =\frac{\partial_t(q_1)}{q_1} \neq 0$\footnote{ Indeed $|q_1| <1$ so that the valuation of $q_1$ at $t=0$ is strictly positive and $\partial_t(q_1) \neq 0$.} A contradiction.  

Therefore, we conclude that \eqref{eq:minimalliaison3} is with coefficients in $C$. Using the commutativity of $\sigma_{q_1}$ with $\partial_s$, we find that for any non-negative integer $k$, the function $\partial_s^k(e)$ is also solution of \eqref{eq:minimalliaison3}. Now set $W=\{ h \in \cM er(C^*) | c_0 h + \dots + c_{r-1} \sigma_{q_1}^{r-1}(h) + \sigma_{q_1}^r(h)=0 \}$. Then $W$ is a $C_{q_1}$-vector space. Replacing a Wronskian argument by a Casoratian argument, one can prove that the dimension of $W$ as $C_{q_1}$-vector space is bounded by $r$. Since  the function $\partial_s^k(e)$ belongs to $W$ for any $k$ in $\N$, we find a non-trivial $C_{q_1}$-liaison among the $(\partial_s^k(e))_{k \in \N}$. This proves that $e$ is $\partial_s$-finite over $C_{q_1}$. Since $\tF^2-e =\widetilde{g}(s) \in C_{q_1}$, the function $\tF^2(s)$ is also $\partial_s$-finite over $C_{q_1}$.  Lemma~\ref{lemma:sdfiniteimpleisfinitegroup} implies that the group of the walk is finite. A contradiction. 
\end{proof}

The study of the generating series $Q(x,0,t)$ is entirely symmetric so that one can combine Lemma~\ref{lemma:sdfiniteimpleisfinitegroup} and Theorem~\ref{thm:tdf} to prove  the following result. 

\begin{thm}\label{thm:main}
If the group of the walk is infinite, the generating series $Q(x,y,t)$ is neither $\frac{\partial}{\partial x}, \frac{\partial}{\partial y} $-finite nor $\frac{\partial}{\partial t}$-finite over $\Q(x,y,t)$. 
\end{thm}

\appendix 
\section{Poles and decoupling}
In this section, we  prove a  result concerning the poles of a decoupling (see lemma~\ref{lem:polesdecoupling} below). Our proof relies on some basic notions in algebraic geometry which we shall recall first. 

In the notation of Section~\ref{subsec:kernelcurve}, we  only consider models for which the kernel curve $E$ is irreducible,  smooth and of genus one.  We consider the base points of the kernel curve $E$, that is,    the common zeros of $ x_0x_1y_0y_1$ and 
$\sum_{(i,j) \in \mathcal{S}} d_{i-1,j-1} x_0^{i} x_1^{2-i}y_0^j y_1^{2-j}=0$  in $\P1(C) \times \P1(C)$. Since these two equations have coefficients in $\Q$, their common zeroes  have coordinates in $\overline{\Q}$, the algebraic closure of $\Q$. Moreover,  the curve $E$ being  irreducible, there are only $8$ base points counted with multiplicities which are represented in   Figure \ref{fig:positionbasepoints}.

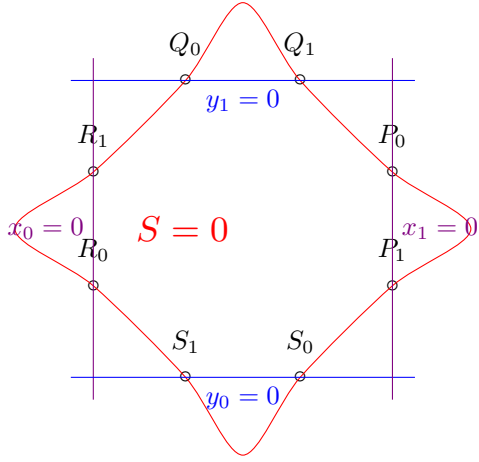
\begin{figure}[h!]

  \begin{tikzpicture}[scale=1.5]
\coordinate (O) at (0:0);
\foreach \i in {1,...,4}{
\coordinate (A\i) at (\i * 90 -4 -45:2);
\coordinate (B\i) at (\i * 90 +94 -45 :2);
 \tkzDefBarycentricPoint(A\i=1,B\i=2)
  \tkzGetPoint{I\i};
   \tkzDefBarycentricPoint(A\i=2,B\i=1)
  \tkzGetPoint{J\i};
}
\coordinate (C1) at (0, 2);
\coordinate (C2) at (-2, 0);
\coordinate (C3) at (0, -2);
\coordinate (C4) at (2, 0);
 
\node[label=above: $Q_0$]  at (I1) {$\circ$};
   \node[label=above: $R_0$]  at (I2) {$\circ$};
    \node[label=above: $S_0$]  at (I3) {$\circ$};
     \node[label=above: $P_0$]  at (I4) {$\circ$};
     
      \node[label=above: $Q_1$]  at (J1) {$\circ$};
   \node[label=above: $R_1$]  at (J2) {$\circ$};
    \node[label=above: $S_1$]  at (J3) {$\circ$};
     \node[label=above: $P_1$]  at (J4) {$\circ$};
     \draw[red] plot[smooth cycle] coordinates {(C1) (I1) (J2) (C2) (I2) (J3) (C3) (I3) (J4) (C4) (I4) (J1)}node[pos=0.5,  left ]{\small{$S=0$}}  ;
     
\draw[blue] (A1)  -- (B1)  node[pos=0.5, below  ] {\textcolor{blue}{$y_1=0$}};
%\draw (A1) to["0"] (B1);
\draw[violet ](A2) --  (B2) node[pos=0.5, left ] {\textcolor{violet}{$x_0=0$}};
%\draw (A2) to["0"] (B2);
\draw[blue] (A3) --(B3) node[pos=0.5, below ] {\textcolor{blue}{$y_0=0$}} ;
%\draw (A3) to["0"] (B3);
\draw[violet] (A4) --(B4)  node[pos=0.5,  right ] {\textcolor{violet}{$x_1=0$}};
%\draw (A4) to["0"] (B4);
\end{tikzpicture}

\caption{Position of the base points } \label{fig:positionbasepoints}
 \end{figure}
 
 Any \emph{regular  fraction} $h(x,y)$ in $C(x,y)$  defines a  \emph{rational function } from  the kernel curve $E$  to $\P1(C)$ by setting $([x_0:x_1],[y_0:y_1]) \mapsto h(\frac{x_0}{x_1},\frac{y_0}{y_1})$. Any  multiple of the kernel polynomial $K(x,y)$ by a regular fraction is sent to the zero morphism so that one can  identify  the function field $C(E)$ of the curve to the field of rational functions from $E$ to $\P1(C)$. \\
 
 A \emph{zero} (resp. a \emph{pole}s) of the function  $h(x,y)$ in $E$  is a point $P=([\alpha_0:\alpha_1],[\beta_0:\beta_1])$ of $E$ such that $h([\alpha_0:\alpha_1],[\beta_0:\beta_1])= 0$ (resp. $\infty$).  The field  $C$  is algebraically closed and contains $\Q(t)$. Since the poles and zeroes (in $C$) of any rational fraction $f(x)$ in  $\Q(x,t)$ are algebraic over $\Q(t)$, the field $C$ contains all the poles of $f(x)$.  If $\alpha$  is a pole in $C \cup \{ \infty\}$ of the  rational fraction  $f(x)$ then the points of $E$ of the form  $( \alpha, [\beta_0:\beta_1])$ are poles of $f(x)$ viewed as a  rational function on $E$.  Conversely, any point $([\alpha_0: \alpha_1], [\beta_0:\beta_1])$ of $E$ that is a pole of $f(x)$ viewed as a rational function on $E$ is such that $\frac{\alpha_0}{\alpha_1}$ is a pole of the \emph{rational fraction} $f(x)$. \\
  
   In order to avoid confusion, we will talk of the poles (in $C$)  of the \emph{rational fraction} $f(x)$ and the poles (in $E$) of the \emph{function} $f(x)$. We'll use the same terminology for elements of $\Q(y,t)$. We know briefly recall some results concerning local parameters on algebraic curves. \\
 
The curve $E$ is smooth. Therefore, for any point $P$ of $E$, there exists a function  $z$ in $C(x,y)$ such that $z$ vanishes at $P$ and every non-zero element $h(x,y)$ in $C(x,y)$ can be written $uz^k$ with $u$ in $C(x,y)$ well defined and non-vanishing  at $P$ and $k$ in $\Z$.  The function $h$ has a pole (resp. a zero) at $P$ if and only if $k$ is strictly negative (resp. strictly positive) and in that case $|k|$ is the order of the pole (resp. zero) of $h$  at $P$. Such a function $z$ is called a \emph{local parameter} and two local parameters $z$ and $z'$ are related by $z'=uz$ with   $u$ in $C(x,y)$ well defined and non-vanishing  at $P$ (see \cite[Theorem 1.1]{Shafarevich}).

\begin{ex}\label{exa:localparam}
If $Q_1=(\alpha_1,\infty)$ is one of the base points then the function $y_1=\frac{1}{y}$ is a local parameter at $Q_1$ if and only if $Q_1 \neq Q_0$. Indeed, $\frac{1}{y}$ vanishes at $Q_1$ and it is  a local parameter at $Q_1$ if and only if  $\partial_{x_1}\widetilde{K}(1,x_1,1,y_1)$ does not vanish at $Q_1$. This condition is equivalent to the fact that $Q_1 \neq Q_0$ (see the first paragraph after \cite[Theorem 1.1]{Shafarevich}). \end{ex}

The main result of this appendix is as follows. 
\begin{lemma}\label{lem:polesdecoupling}
Assume that the curve $E$ is of genus one and that the group of the walk is infinite.  Let $f(x) \in \Q(x,t)$ and $g(y) \in \Q(y,t)$ be a decoupling of $xy$. Then,  $f(x)$ has a pole in the closed unit disk $\{x \in C | |x| \leq 1\}$ or $g(y)$ has a pole in the closed unit disk $\{ x \in C |  |y| \leq 1\}$.
\end{lemma}
\begin{proof}
 The poles  in $E$ of the function $xy$ belong to the set $\{ P_0,P_1,Q_0,Q_1 \}$. Because of the decoupling equation,   the   function $xy$ coincides with  the  function $f(x)+g(y)$ on $E$. Thus, the set of poles of $f(x) +g(y)$ coincides with the set of poles of $xy$.  
By \cite[Proposition 4.3]{HardouinSingerSelecta},  if $xy$ has a decoupling then there is no $P_i$ and $Q_j$ simultaneously fixed by an involution $\iota_1$ or $\iota_2$. A direct corollary  is that if $Q_0 =Q_1$ then $P_0 \neq P_1$.  \\

 We shall thus prove that if $Q_0 \neq Q_1$ then $f(x)$ has a pole in the closed unit disk $\{x \in C \, | \, |x| \leq 1\}$ or $g(y)$ has a pole in the closed unit disk $\{y \in C \,  | \,  |y| \leq 1\}$. The case where $P_0 \neq P_1$ being symmetric by using the $x$ and $y$ symmetry, this will be enough to conclude the proof.  \\
  
Assume from now on that $Q_0 \neq Q_1$. Then,  the function $\frac{1}{y}$ is a local parameter at $Q_0$ and $Q_1$ (see Example~\ref{exa:localparam}).  We divide the  proof in three cases. \\

\textbf{Case 1: $Q_0$ and $Q_1$ do not belong to $\{(0, \infty), (\infty, \infty) \}$.} \\

Then, $Q_i= (\alpha_i,\infty)$ with $\alpha_i \in \overline{\Q}^*$ and $\alpha_0 \neq \alpha_1$. Thus, the function $xy$ has a pole of order $1$ at each $Q_i$. Indeed, $x$ is non-vanishing and well-defined at each of the $Q_i$'s and $\frac{1}{y}$ is a local parameter at each $Q_i$'s. 

 Let us suppose that none of the $Q_i$'s is a pole of the function $f(x)$. Then, $Q_0$ and $Q_1$ must be  poles of order $1$ of the function  $g(y)$.  Let us write $g(y)$ as $P(y) +h(y)$ where $P(y)$ is  the polynomial part of the  rational fraction $g(y)$ so that the fraction $h(y)$ has no pole at $y=\infty$. Since $1/y$ is a  local parameter  at each  $Q_i$ and $g(y)$ has a pole of order one at each $Q_i$,  the polynomial part $P(y)$ of $g(y)$ must be of degree $1$ and we write it as $P(y)=\alpha y+\beta$ for some $(\alpha,\beta) \in \C^* \times \C$.

 Now, we consider the equality $xy=f(x)+g(y)$ locally at each of the $Q_i$'s. Since $x-\alpha_i$ vanishes at $Q_i$, one can write $x-\alpha_i=u_i\left(\frac{1}{y}\right)^{k_i}$ for $k_i>0$ and $u_i$ well-defined and non-vanishing at $Q_i$.    For $i=0,1$, one has   $xy= \alpha_i y + l_i$ where $l_i =u_i\left(\frac{1}{y}\right)^{k_i}$  is a function with no poles at $Q_i$.
 
 Since  the $Q_i$'s are not poles of $f(x)$ nor of $h(y)$,  one can write  $f(x) +g(y)=\alpha y + \widetilde{l}_i$ where $\widetilde{l}_i=f(x)+ h(y) +\beta$ is a rational function on $E$ with no poles at the $Q_i$'s. Since the function $xy$ coincides with $f(x)+g(y)$ at each of the $Q_i$'s, we must have $\alpha=\alpha_0=\alpha_1$. A contradiction.  Then, one of the $Q_i$'s, say $Q_1$ is a pole of $f(x)$. This is possible only if the element $\alpha_1$ is a pole of the rational fraction $f(x)$. To conclude, note that any non-zero element of $\overline{\Q}$ has norm $1$ and that  $|0|=0$. \\

\textbf{Case 2: one of the $Q_i$'s equals $(0,\infty)$.} \\

Assume that $Q_1=(0,\infty)$. Then, since $Q_1 \neq Q_0$,  the point $Q_0$ is a pole of $xy$. Moreover, the point  $Q_1$ is not a pole of the function $xy$.  Indeed,  the function $\frac{1}{y}$ is a local parameter at $Q_1$. The function $x$ vanishes at $Q_1$ so that $x=u \left(\frac{1}{y}\right)^k$ for $k>0$ and $u$ well-defined and non vanishing at $Q_1$. Thus, $xy= u \left(\frac{1}{y}\right)^{k-1}$ is well-defined  at $Q_1$.

If $Q_0$ is a pole of the function $g(y)$ then the same holds for $Q_1$.  Since $Q_1$ is not  a pole of $xy$ then it must be a pole of the function $f(x)$. Then, $0$ is a pole of the rational fraction $f(x)$  and $|0|=0 \leq 1$.

 If   $Q_0$ is not a pole of  the function  $g(y)$ then it   must be  a pole of $f(x)$. If $Q_0=(\alpha_0,\infty)$, we can conclude as above that $\alpha_0$ is a pole of $f(x)$ with $|\alpha_0|=1$.  If $Q_0=(\infty,\infty)$ then it is a pole of order $2$ of $xy$. Indeed, since $Q_0$ coincides with one of the $P_i$'s say $P_0$, \cite[Proposition 4.3]{HardouinSingerSelecta}  yields $P_0 \neq P_1 =(\infty,\beta_1)$ for some non-zero $\beta_1$ in $\overline{\Q}$.
 Then, the function $\frac{1}{x}$ is a local parameter at $Q_0=P_0$ by an argument analogous to Example~\ref{exa:localparam}. Thus, $\frac{1}{y}=u\frac{1}{x}$  for some  well-defined and non-vanishing at   $Q_0$ function $u$. Thus,   $xy=uy^2$ has a pole of order $2$ at $Q_0=P_0$.   Since $Q_0$ is not a pole of $g(y)$, it  must be a pole of order $2$ of $f(x)$ so that   the polynomial part of $f(x)$ must have degree $2$ (because $\frac{1}{x}$ is a local parameter at $Q_0$).  Thus, $P_1$ is a pole of order $2$ of $f(x)$ as well. But, the function $xy$ has only a pole of order $1$ at $P_1$ (because $\beta_1 \neq 0$ and $\frac{1}{x}$ is a local parameter at $P_1$).  Since $xy=f(x) +g(y)$, the point $P_1$ must be a pole of order $2$ of the function $g(y)$. This proves that $\beta_1$ is a pole of the rational fraction $g(y)$. One concludes by noting that  $|\beta_1|=1$ because $\beta_1$ is in $\overline{\Q}^*$. \\

\textbf{Case 3: one of the $Q_i$'s equals $(\infty,\infty)$.} \\

One can assume that $P_0=Q_0=(\infty,\infty)$. As above, \cite[Proposition 4.3]{HardouinSingerSelecta} shows that $P_1=(\infty,\beta_1)$ and $Q_1=(\alpha_1,\infty)$ with $\alpha_1,\beta_1$ in $\overline{\Q}^*$. Reasoning as above,   we find that  $xy$ has a pole of order $2$ at $Q_0$.

 Since $x$ is well defined and non-vanishing at $Q_1$ and $\frac{1}{y}$ is a local parameter at $Q_1$, the function $xy$   has  a pole of order $1$ at $Q_1$. If $Q_1$ is a pole of $f(x)$ then we are done since $|\alpha_1|=1$. If $Q_1$ is not a pole of $f(x)$ then it must be a pole of order $1$ of $g(y)$ which implies that the degree of the polynomial part of $g(y)$ is $1$. This implies that $Q_0$ is a pole of order $1$ of $g(y)$ as well. Since $Q_0=P_0$ is a pole of order $2$ of $xy$ and thereby of $f(x) +g(y)$, one concludes that $Q_0$ is  a pole of order $2$ of $f(x)$. This shows that $P_1$ is a pole of order $2$ of the function $f(x)$. However, since $y$ is well-defined and non-vanishing at  $P_1$ and $\frac{1}{x}$ is a local parameter at $P_1$, the function $xy$ has a pole of order one at $P_1$. Thus, $P_1$ must be  a pole of order $2$ of the function $g(y)$. This proves that $\beta_1$ is a pole of $g(y)$ viewed as  a rational fraction. One concludes by noting that  $|\beta_1|=1$ because $\beta_1$ is in $\overline{\Q}^*$.
\end{proof}

\bibliography{walkbib}

\bibliographystyle{alpha}

\end{document}